\documentclass[11 pt]{article}
\usepackage{eurosym}
\usepackage{stmaryrd}
\usepackage{amssymb}
\usepackage{amsmath}
\usepackage{xcolor}
\usepackage{amsthm}
\usepackage{amscd}
\usepackage{graphicx}
\usepackage{hyperref}
\usepackage{geometry}
\usepackage[all]{xy}
\usepackage{mathpazo}
\usepackage{avant}
\usepackage{tikz,xcolor,hyperref}

\definecolor{lime}{HTML}{A6CE39}
\DeclareRobustCommand{\orcidicon}{
	\begin{tikzpicture}
	\draw[lime, fill=lime] (0,0) 
	circle [radius=0.16] 
	node[white] {{\fontfamily{qag}\selectfont \tiny ID}};
	\draw[white, fill=white] (-0.0625,0.095) 
	circle [radius=0.007];
	\end{tikzpicture}
	\hspace{-2mm}
}
\foreach \x in {A, ..., Z}{\expandafter\xdef\csname orcid\x\endcsname{\noexpand\href{https://orcid.org/\csname orcidauthor\x\endcsname}
			{\noexpand\orcidicon}}
}

\allowdisplaybreaks

\def\vbar{\mathchoice{\vrule height6.3ptdepth-.5ptwidth.8pt\kern- .8pt}
{\vrule height6.3ptdepth-.5ptwidth.8pt\kern-.8pt} {\vrule
height4.1ptdepth-.35ptwidth.6pt\kern-.6pt} {\vrule
height3.1ptdepth-.25ptwidth.5pt\kern-.5pt}}

\def\<{\langle}
\def\>{\rangle}

\newtheorem{theorem}{Theorem}[section]
\newtheorem{lemma}[theorem]{Lemma}
\newtheorem{corollary}[theorem]{Corollary}
\newtheorem{proposition}[theorem]{Proposition}
\newtheorem{example}[theorem]{Example}

\theoremstyle{definition}
\newtheorem{definition}{Definition}[section]
\theoremstyle{remark}
\newtheorem{remark}{Remark}[section]

\begin{document}

\title{Common Fixed Point of the Commutative F-contraction Self-mappings with uniquely bounded sequence }

\author{\textbf{Djamel Deghoul $^{1}$\thanks{E-mail: djamel.deghoul@univ-setif.dz}\,,
Zoheir Chebel\orcidZ{}$^{2,3}$\thanks{E-mail: zoheir\_chebel1@yahoo.fr}\,,
Abdellatif Boureghda\orcidA{} $^{1}$\thanks{Corresponding author, E-mail: abdellatif.boureghda@univ-setif.dz; abdellatif.boureghda@uha.fr}\;, 
Salah Benyoucef $^{4}$ \thanks{E-mail: saben21@yahoo.fr}}\\\\
{\small 1. Laboratory of Fundamental and Numerical Mathematics, Faculty of Sciences, Ferhat Abbas University Sétif 19137, Algeria.}\\
{\small 2. Department of Mathematics, Si El Haoues University Center Barika, Algeria.}\\
{\small 3. Laboratory of Advanced Electronics and Telecommunications, El Bachir El Ibrahimi University Bordj Bou Arreridj, Algeria.}\\ 
{\small 4. Department of Mathematics University of 20 août 1955 Skikda, Algeria.}}
\date{}
\maketitle

\begin{quote}
\textbf{Abstract:} 
We establish the existence of a common fixed point for mappings that satisfy and extend the F-contraction condition. To support our findings, we present pertinent definitions and properties associated with F-contraction mappings. Additionally, we establish an analogue to the Banach contraction theorem. Our results contribute to the broader understanding of this field by extending and generalizing existing findings in the literature.


\end{quote}


\section{Introduction}
\label{sec-1}
In 1976, Jungck \cite{G} pioneered the proof that, when provided with two continuous functions, $f$ and $g$, defined on a complete metric space, with the additional property of being commuting functions and $g$ being an F-contraction such that, the range of $g$ be included in the one's of $f.$
Then $f$ and $g$ must posses a unique common fixed point. It is essential to note that the inclusion condition in Jungck's theorem is sufficient but not necessary for the existence of common fixed points.
In this current study, we maintain all the aforementioned conditions while substituting the inclusion criterion with the requirement of a bounded Picard sequence. This adjustment is demonstrated to be both necessary and sufficient for Jungck's result to hold true.
\section{preliminaries}
Allow us to revisit certain definitions and established results pertaining to common fixed points.

\begin{definition}\cite{chebel}
Consider two metric spaces, $(X , d)$ and $(X^{\prime}, d^{\prime})$  equipped with distances $d$ and $ d^{\prime}$ respectively. Let $f$ and $g$ be mappings defined as $ f , g : X \longrightarrow X^{\prime}$. A mapping $g$ is deemed an F-contraction, if there exists a real constant $ 0< k < 1$, satisfying the condition: 
\begin{equation}\label{*}
d^{\prime}(g(x), g(y)) \leq k d^{\prime}( f (x), f (y)), \forall x, y \in X .
\end{equation}
We denote this relationship as $g - k - f$ If $f$ is continuous, we refer to $g - k - f$ as a continuous contraction. 
\end{definition}

The objective of this work is to provide a generalization of the following theorem by eliminating the convergent condition from the hypothesis and streamlining the assumptions.

\begin{theorem}\cite{chebel}
Consider a continuous contraction mapping $g - k - f$ in the complete metric space $X$. Additionally, assume that the mappings $f$ and $g$ commute, and there exists an element $x_{0} \in X $ such that the Picard sequence $\{ f_{n} (x_{0})\}_{n\geq 0}$ converges to $t_{0} \in X $. In this case, the Picard sequence $\{g_{n} (x_{0})\}_{n\leq 0}$ converges to a point $r \in X$. Furthermore, the Picard sequence $ \{f_{n}(r )\}_{n\geq 0}$ is bounded, then the sequence $\{g_{n} (t)\}_{n\geq 0}$ converges to $r$ which stands as the unique common fixed point of the mappings $f$ and $g.$  
\end{theorem}


\begin{remark}
If we substitute the mapping $f$ with an identity mapping in the condition \ref{1}, we retrieve the classical contraction mapping scenario, leading us to the well-known Banach fixed-point theorem, as discussed in \cite{Gustave}.
\end{remark}

\section{Main results and theorems}

The proof of our theorems relies on the establishment of certain definitions, properties, propositions, and lemmas.

 In what follows, let $x_{0}$ denote an element of a non-empty complete metric space $X$.
For the sake of notation, we introduce $g^{0}(x_{0})=x_{0}, f^{0}(x_{0}))=x_{0}$ and
inductively  $g^{n+1}(x_{0})=g(g^{n}(x_{0})), f^{n+1}(x_{0})=f(f^{n}(x_{0}))$, where $ n\in\{1, 2, ...\}.$

The results presented in this subsection are commonly referred to as a variant of Banach's contraction principle.
\begin{theorem}\label{theorem principal}
Let $f$ and $g$ be two commuting self-mappings on the complete metric space $X$. Assume that the maps $f$ and $g$ satisfy the condition \ref{*} with $f$ being continuous. If there exists an element $x_{0}\in X$ such that the sequence $\{f^{n}(x_{0})\}_{n\geq 0}$ is bounded, then the maps $f$ and $g$ possess a unique common fixed point in $X$.
\end{theorem}

To establish the proof of the main theorem, we require the following lemmas in succession.

\begin{lemma}\label{1}
Let $f$ and $g$ be two  self-mappings commuting with each other. Assuming that $f$ and $g$ are satisfy condition \eqref{*}, the following inequality holds.
\begin{equation}\label{eq1}
d(g^{n}(x), g^{n}(y))\leq k^{n} d(f^{n}(x),f^{n}(y)),\ \ \forall x,y \in X, \forall n\in\mathbb{N}.
\end{equation}
\end{lemma}

\begin{proof}
Consider two elements $x, y $ in the space $X$. From the condition \ref{*}, then,
$$ d(g\circ g^{n-1}(x), g\circ g^{n-1}(y))\leq k d(f\circ g^{n-1}(x), f\circ g^{n-1}(y)).$$

Utilizing the commutativity, we observe that
\begin{equation}\label{A11}
f\circ g^{n}=g^{n}\circ f, \forall n\in \mathbb{N}.
\end{equation}
The same line of reasoning yields:
$$ d(f\circ g^{n-1}(x), f\circ g^{n-1}(y))= d(g^{n-1}\circ f(x), g^{n-1}\circ f(y))$$. Applying conditions \ref{A11} and \ref{*}, once again, we obtain:
$$ d(g^{n-1}\circ f(x), g^{n-1}\circ f(y))\leq k d(g^{n-2}\circ f^{2}(x), g^{n-2}\circ f^{2}(y)).$$ Through induction, we derive:
$$ d(g^{n}(x), g^{n}(y))\leq k^{n} d(f^{n}(x),f^{n}(y)). $$
This concludes our proof.
\end{proof}

\begin{lemma}\label{2}
Consider two commuting maps  $f$ and $g$ satisfying condition \ref{*}. The ensuing inequality is established:.
\begin{equation}\label{eq1}
d((f\circ g)^{n}(x), (f\circ g)^{n-1}(y))\leq k^{n-1} d(g\circ f^{2n-1}(x), f^{2n-2}(y)) \leq s k^{n-1}, \forall x,y \in X.
\end{equation}
\end{lemma}

\begin{proof}
Let $x_{0}$ be an elements in $X$. Applying the lemma\ref{1}, we deduce.

\begin{eqnarray}
d((f\circ g)^{n}(x_{0}), (f\circ g)^{n-1}(x_{0})) &=& d(g^{n-1}(g\circ f^{n})(x_{0})), g^{n-1}( f^{n-1}(x_{0}))) \nonumber \\
                                           &\leq& k^{n-1} d(g\circ f^{2n-1}(x_{0}), f^{2n-2}(x_{0})) \nonumber \\
                                           &\leq& s k^{n-1}. \nonumber
\end{eqnarray}
By virtue of condition \ref{*} and boundedness of  $\{f^{n}(x_{0})\}_{n\geq 0}$, we deduce that the distance $d(g\circ f^{2n-1}(x), f^{2n-2}(y)$ is also bounded.
Let $s=\sup_{n\geq 1} d(g\circ f^{2n-1}(x_{0}), f^{2n-2}(x_{0}))$
This completes the proof.

\end{proof}

\begin{lemma} \label{3}
Under the hypothesis of the theorem \ref{theorem principal} the sequence $ \{(f\circ g)^{n}(x_{0})\}_{n\geq 0}$  converges in the complete metric space $X.$
\end{lemma}

\begin{proof}
To establish the convergence, we first demonstrate that the sequence is a Cauchy sequence. Consider integers $n, m$, where $n > m$. By employing the triangle inequality and Lemma\ref{2}, we find
\begin{eqnarray}
d((f\circ g)^{n}(x_{0}), (f\circ g)^{m}(x_{0}))&\leq& \sum_{j=m}^{j=n-1}d((f\circ g)^{j}(x_{0}), (f\circ g)^{j+1}(x_{0})). \nonumber \\
                                               & \leq & \sum_{j=m}^{j=n-1}s k^{j}.\nonumber \\
                                               &\leq& s \frac{k^{m}-k^{n}}{1-k}.
\end{eqnarray}
By allowing $n\rightarrow \infty$ and $m\rightarrow \infty$, the sequence becomes a Cauchy sequence. Therefore, by completeness, it converges to a limit $l$.
Leveraging the continuity of the maps $f$ and $g,$ we observe that $f\circ g(l)=l$. This completes the proof of the lemma.
\end{proof}

\begin{corollary}
Under the hypothesis of theorem \ref{theorem principal}, If the sequence $ \{f^{n}(x_{0})\}_{n\geq 0}$ is bounded, it follows that the sequence $ \{(f\circ g)^{n}(x_{0})\}_{n\geq 0}$ is also bounded in the metric space $X$.
\end{corollary}

\begin{lemma}\label{4}
Consider the hypothesis of lemma \ref{2} which asserts the existence of a constant $c$ such that:
\begin{eqnarray}\label{eq}
\forall n, m \in \mathbb{N}, n > m,\ \ d(g^{n}(l), g^{m}(l))\leq c \frac{k^{\frac{m}{2}}-k^{\frac{n}{2}}}{1-\sqrt{k}}.
\end{eqnarray}
Additionally, it is established that the sequence $\{g^{n}(l)\}_{n\geq 0}$ is convergent.
\end{lemma}

\begin{proof}
Let $n$ be a positive integer. It is straightforward to verify the following inequality:\\

\begin{eqnarray}
d(g^{n}(l), g^{n-1}(l))&=& d(g\circ g^{n-1}(l), g\circ g^{n-2}(l)). \nonumber\\ \textit{By the condition \ref{*} }
&\leq & k d(f\circ g^{n-1}(l), f\circ g^{n-2}(l)). \nonumber\\ \ \ \ \textit{By exploiting the commutativity  }
&\leq & k d( g^{n-2}(g\circ f (l)), g^{n-3} (g\circ f(l))), \nonumber\\ \ \ \  \textit{and by lemma \ref{3}}
&\leq & k d( g^{n-2}(l), g^{n-3}(l)). \nonumber
\end{eqnarray}
and employing induction, we distinguish two cases: \\
1. If $n$ is odd, then: \\
$$d( g^{n}(l), g^{n-1}(l))\leq k^{\frac{n-1}{2}} d(g(l), l)$$
2. If n is even, then
$$d( g^{n}(l), g^{n-1}(l))\leq k^{\frac{n}{2}} d(f(l), l).$$
Set $$ c=\max \{d(g(l), l), \sqrt{k} \ d(f(l), l)\}.$$
It follows: \\
$$d( g^{n}(l), g^{n-1}(l)) \leq c k^{\frac{n-1}{2}}. $$
For any positive integers $n$ and $m$ with $n>m$, we have
\begin{eqnarray}
\forall n > m,\ \ d(g^{n}(l), g^{m}(l))\leq \sum_{j=m}^{n-1} d(g^{j}(l), g^{j+1}(l))&\leq& c \sum_{j=m}^{n-1} k^{\frac{j}{2}}, \nonumber \\
                                                                                      &\leq& c \frac{k^{\frac{m}{2}}-k^{\frac{n}{2}}}{1-\sqrt{k}}. \nonumber \\
\end{eqnarray}
By letting $n\rightarrow \infty$ and $m\rightarrow \infty,$ we obtain $ d(g^{n}(l), g^{m}(l))\rightarrow 0,$ therefore the sequence $\{g^{n}(l)\}_{n\geq 0}$ is a Cauchy sequence and, by completeness, converges to a limit $l_{1}$. The continuity argument of $g$ yields $g(l_{1})=l_{1}$,
concluding the proof of the lemma.
\end{proof}

\begin{lemma} \label{5}
For every integer $n \geq 1,$ the following equality is satisfied.\\
$(g \circ f )(g^{n}(l))=g^{n}(l)$.
\end{lemma}

\begin{proof}
Straightforward.
\end{proof}

\begin{lemma}\label{6}
Under the conditions of theorem \ref{theorem principal}, the composed mapping $g\circ f$  possesses a fixed point in the space $X.$
\end{lemma}

\begin{proof}
Utilising limit lemmas \ref{4} and \ref{5}, we establish that  the sequence $(g\circ f) (g^{n}(l))$ is a Cauchy sequence in the complete space $X.$ Consequently, it converges to the limit $ l_{1}\in X.$
Taking the limit and exploiting the continuity of $g\circ f,$ we deduce that: $$ \lim\limits_{n \to +\infty} (g\circ f) (g^{n})(l)=l_{1}=g\circ f(l_{1}).$$
This concludes the proof.
\end{proof}

Now, we proceed to prove the main result of Theorem \ref{theorem principal}.

\begin{proof}
Proof of theorem \ref{theorem principal}.\\
 The necessity of the condition is evident. If $f$ and $g$ have a common fixed point $t$ then the sequence $(f^{n}(t))_{n\geq 0}$ is bounded.
 Regarding sufficiency, suppose that there exists an element $x_{0}\in X,$ such that the sequence $(f^{n}(x_{0}))$ is bounded. It is easy to see that $l_{1}$ is a common fixed point of the maps $f$ and $g.$
In fact, by lemmas \ref{4}, \ref{6} and the commutativity, we have:
 $$ g\circ f(l_{1})=f (g(l_{1}))=f(l_{1})=g(l_{1})=l_{1}. $$

Next, we prove the uniqueness of the common fixed point of $f$ and $g$. \\
Let $l_{1}$ and $l_{2}$ be two common fixed points of $f$ and $g$ i.e $f(l_{1})=g(l_{1})=l_{1}$ and  $f(l_{2})=g(l_{2})=l_{2}.$  By computing the distance between $l_{1}$ and $l_{2}$, we obtain,
$$d(l_{1}, l_{2})= d(g(l_{1}), g(l_{2}))\leq k d(f(l_{1}),f(l_{2}))= k d(l_{1}, l_{2}).$$
From the last inequality, we deduce that, $(1-k)d(l_{1}, l_{2})\leq 0.$  Since $k<1,$ then $l_{1}=l_{2}.$ This shows the uniqueness of the common fixed point and concludes the proof of our theorem.
\end{proof}

\begin{example}
Let $X =[1, +\infty[$ with the usual metric. For integer number $p, q $ such that $p<q$, define $f,g : X\rightarrow X $,$f(x) =x^{p}\ \textit{and} \ g(x) =x^{q}.$ \\
All conditions of theorem  are fulfilled. In fact, $f $ and $g$ commute, the contraction is given by $ |f(x)-f(y)|\leq \frac{p}{q}|g(x)-g(y)|, \forall x,y \in X $ and finally, the iterations of the Picard sequence $g^{n}(1)=1$ are bounded at the point $x_{0}=1.$ In this situation, the unique common fixed point of $f$ and $g$ is the element $1.$

\end{example}

\begin{remark}
The following corollary is an example illustrating that the inclusion condition in theorem 1.1 in \cite{G} and the boundedness condition in our theorem are independent of each other.
\end{remark}

\begin{corollary}
For positive integers $n, m\geq 1$ and a positive real number  $ k>1$, let $f$ be a continuous map defined on the complete metric space $X$ into itself.  Assume that the inequality
\begin{equation}
d(f^{n}(x),f^{n}(y))\geq k d(x,y), \textit{holds.}
\end{equation}
The map $f$ has a unique fixed point in $X$ if and only if there exists an element $x_{0}\in X $ such that the sequence $\{f^{m}(x_{0})\}_{m\geq 1}$ is bounded.
\end{corollary}

\begin{remark}
Here, the map $f$ need not be onto. However in corollary 2 of reference \cite{G}, $f$ must be onto.
\end{remark}

\begin{corollary}
Let f and g be commuting mappings of a complete metric space $(X, d)$ into itself where $f$ is continuous. Consider positive integers $ m, n, p $. Suppose the inequality
\begin{equation}
d(g^{m}(x),g^{m}(y)) \leq k d(f^{n}(x),f^{n}(y)), \ \ 0<k<1,\ \ \textit{holds.}
\end{equation}
Then, the maps $f$ and $g$ have a unique common fixed point if and only if there exists an element $x_{0}\in X$ such that the sequence $\{f^{p}(x_{0})\}_{p\geq 0}$ is bounded.
\end{corollary}

The subsequent corollaries deal with the existence of a common fixed point for three commutative self-mappings.

\begin{corollary}\label{final}
Let $f,$ $g$ and $h$ be three continuous and commuting mappings defined on the complete metric space into itself. Assume the condition
\begin{equation}\label{equation2}
d(g(x),g(y)) \leq k d(h(x), h(y)),\  \forall x, y \in X, \ \ \textit{holds.}
\end{equation}
The maps $ f, g $ and $h$ have a unique common fixed point in the subspace $X,$ if and only if there exists an element $x_{0}\in X$ such that the sequence $\{h^{n}(x_{0})\}_{n\geq 0}$ is bounded.
\end{corollary}

\begin{proof}
Utilizing Theorem \eqref{theorem principal} and the uniqueness of the common fixed point.
\end{proof}

\begin{corollary}
Under the assumption of corollary \ref{final}, if the map $h$ is a bounded one mapping, then the maps $f,$ $g$, and $h$ have a unique common fixed point in $X$. if the map $h$ a bounded one mapping.
\end{corollary}

In the following, we explore connections with the Theorem 1.1 [1] and Theorem 3.8

\begin{theorem}
Let $h - k - g$ and $g - kk^{\prime} - f$ be two contraction mappings in the complete metric $X$ into itself. Assume they are continuous commuting with each other, and there exists an element $x_{0} \in  X$ for which the Picard sequence $\{f_{n} (x_{0})\}_{n\geq}$ is bounded in $X.$ Then, the maps $f , g,$ and $h$ have a unique common fixed point in $X.$  

\end{theorem}

\begin{proposition}
Let $g_{1} - k - f_{1}$ and $g_{2} - k^{\prime} - f_{2}$ be two continuous contraction mappings in the compact $M$ into itself, where $f_{1}$ commutes with $g_{2}$. Assume  $g_{1} \circ g_{2}$ and $f_{2} \circ f_{1}$ commute with each other, and  $\{f_{2}\circ f_{1}(x_{0})\}_{n\geq 0}$ is bounded in $X.$ Then, $g_{1} \circ g_{2}$ and $f_{2} \circ f_{1}$ have a unique common fixed point in $X.$ 
\end{proposition}

Lastly, we observe that the requirement \ref{theorem principal} of our theorem can be more weakened by demanding that $(X,d)$ be compact.
\begin{corollary}
 For a positive integer $n\geq 1$ and a positive real number  $k>1$,let $f$ be a continuous map defined on the compact metric space $(X, d)$ into itself.  If the inequality holds, 
\begin{equation}
d(f^{n}(x),f^{n}(y))\geq k d(x,y),
\end{equation}
then the map $f$ has a unique fixed point in $X.$
\end{corollary}

\begin{corollary}
Let $f,$ $g$ and $h$ be three and commuting mappings on the compact space $(M,d)$ into itself. Assume $h$ and $g$ are continuous and satisfy a certain condition \ref{equation 2}. Then,the maps $f$, $g$ and $h$ have a unique common fixed point in $X$.
\end{corollary}

\begin{corollary}
All mappings $f$ commutate with a contraction mapping $g$, defined on the compact metric space $(M,d)$ into itself, have a unique common fixed point.
\end{corollary}

\begin{corollary}
Consider a sequence of continuous functions $f_{n}$  defined on the compact set $K$ converging uniformly to a function $f$. If there exists a contraction mapping $g$ that commutes with each $f_{n}$ for all positive integers $n$, then $f$ and $g$ possess a unique common fixed point.
\end{corollary}

\begin{example}
Let the functions $f$ and $g$ be defined on the interval $[\frac{1}{2}, \frac{3}{4}]$ and map into itself. Specifically,  
$g(x) = xe^{2x+1}$; $f(x) =\frac{-1}{2}x +\frac{3}{4}$ \\
It is evident that $f$ and $g$ commute, as demonstrated by the calculation of the derivative of
$\mid \frac{g^{\prime}(x)}{f^{\prime}(x)} \mid=2\mid e^{1-2x}(2x-1) \mid $. Additionally, the contraction property of $g$ is apparent. Indeed, for all x in the interval ,\\$\forall x \in [\frac{1}{2}, \frac{3}{4}],$ we have: $\mid \frac{g^{\prime}(x)}{f^{\prime}(x)} \mid<2e^{-1}<1$. Therefore, $g$ is an F-contraction. The unique common fixed point of $f$ and $g$ is determined to be $f(\frac{1}{2})=g(\frac{1}{2})=\frac{1}{2}$.

\end{example}

\section{Conclusions}
  This study delves into the identification of common fixed points among commuting mappings, presenting a comprehensive examination of both necessary and sufficient conditions. This innovative approach introduces a paradigm shift, guaranteeing the existence of a shared fixed point. Its versatility spans across various disciplines, encompassing economic sciences, as well as other realms of mathematics and physics, where it can be effectively deployed through numerical programs.
  
\section*{Acknowledgments}
The research outlined in this paper received invaluable support from the research grant provided by Mohamed El Bachir El Ibrahimi University in Borj Bou Arréridj, Algeria, and the Algerian Ministry of Higher Education and Scientific Research. The project, titled "Applied Algebra and Functional Analysis" (Project(PRUF) code: C00L03UN340120230004), was made possible through their generous funding.

The authors express their sincere gratitude to the referees for their meticulous review of the paper and for offering valuable suggestions and comments. Their thoughtful insights significantly contributed to enhancing the quality and rigor of this work.

\clearpage

\end{document}